\documentclass[12pt]{amsart}
\usepackage{amssymb, amsmath, amscd, amsthm}
\usepackage{hyperref}
\usepackage{xcolor}\hypersetup{linktocpage,colorlinks=true,linkcolor=blue,citecolor=blue}
\usepackage[all,cmtip]{xy}
\usepackage{tikz, tikz-cd}
\usetikzlibrary{matrix,arrows}
\usepackage{mathrsfs}
\usepackage{biblatex}
\usepackage{centernot}
\usepackage[margin=1.4in]{geometry}

\renewbibmacro*{in:}{}
\numberwithin{equation}{section}
\allowdisplaybreaks

\theoremstyle{plain}
\newtheorem{theorem}{Theorem}[section]
\newtheorem{lemma}[theorem]{Lemma}
\newtheorem{corollary}[theorem]{Corollary}
\newtheorem{proposition}[theorem]{Proposition}

\theoremstyle{definition}
\newtheorem{definition}[theorem]{Definition}
\newtheorem{example}[theorem]{Example}

\newcommand{\coker}{\mathrm{coker}}
\newcommand{\im}{\mathrm{im}}
\newcommand{\ind}{\mathrm{ind}}
\newcommand{\tr}{\mathrm{tr}}
\newcommand{\ord}{\mathrm{ord}}

\bibliography{refs}

\begin{document}

\title[Functional calculus and joint torsion]{Functional calculus and joint torsion of pairs of almost commuting operators}

\author{Joseph Migler}

\address{Department of Mathematics \\ University of Colorado \\ Boulder, CO 80309-0395}

\email{joseph.migler@colorado.edu}

\date{}

\begin{abstract}
This paper investigates the transformation of determinants of pairs of Fredholm operators with trace class commutators.
We study the extent to which the functional calculus commutes, modulo operator ideals, with projections in a finitely summable Fredholm module.
As an application, we recover in particular some results of R.~Carey and J.~Pincus on determinants and Tate tame symbols.
Additionally, we obtain variational formulas for joint torsion.
\end{abstract}

\subjclass[2010]{Primary 19C20; Secondary 47A13, 47A53, 47A60, 47B35}

\keywords{joint torsion; tame symbols; functional calculus, multivariable operator theory; determinants in K-theory}

\thanks{This work was completed as part of the author's doctoral dissertation at the University of Colorado, under the direction of Professor Alexander Gorokhovsky.}

\maketitle

\thispagestyle{empty}

\tableofcontents

\section{Introduction}

This paper was partly motivated by a desire to understand the works \cite{Joint, Perturbation}.  In particular, R.~Carey and J.~Pincus introduce a type of determinant known as the joint torsion $\tau(A,B)$ associated to any two Fredholm operators $A$ and $B$ which commute modulo the trace ideal $\mathcal L^1$ on a Hilbert space.  They use this invariant to obtain generalizations, for symbols with nonzero winding numbers, of Szeg\H{o}'s limit theorems on the asymptotics of determinants of Toeplitz operators.  It turns out \cite{JT-equals-DT} that joint torsion is equal to the determinant invariant of L.~Brown in algebraic $K$-theory \cite{Brown}.

In \cite{Kaad}, J.~Kaad generalizes the notion of joint torsion to commuting tuples of operators satisfying a natural Fredholm property.  J.~Kaad and R.~Nest have developed a theory of perturbation vectors associated to pairs of complexes which are perturbations of one another \cite{Kaad-Nest2}.  They have also investigated local indices of $n$-tuples of commuting operators under the holomorphic functional calculus \cite{Kaad-Nest}, 
and they obtain a global index theorem originally due to J.~Eschmeier and M.~Putinar \cite
{Eschmeier-Putinar-Book}.
In Section \ref{sec:functional-calculus} we 
establish a multiplicative analogue of such transformation rules in the case of single operators (Proposition \ref{prop:Eschmeier-Putinar-torsion}): the joint torsion $\tau(f(A), B)$ of $f(A)$ and $B$ is
\[
\prod_{ \{ \lambda \in \sigma(A) \, | \, f(\lambda) = 0 \} } \tau(A-\lambda, B)^{\ord_\lambda(f)} \cdot
\tau(q(A), B)
\]
Here $q(A)$ is an invertible operator, so the second factor is a type of multiplicative Lefschetz number.  In addition, we investigate variational formulas for joint torsion 
(Corollaries \ref{cor:variational-exponentials}, \ref{cor:variational-scalar-multiplication}, and \ref{cor:variational-powers}).

We recall the notion of a symbol in arithmetic \cite{Tate}, which is a bimultiplicative map $c(\cdot, \cdot)$ on the multiplicative group of a field such that 
$c(a, 1-a) = 1$.
An example of a symbol is the tame symbol on a field of meromorphic functions, defined as a weighted ratio of the functions (Definition \ref{def:tame-symbol}).
It turns out that this tame symbol is closely related to the Steinberg symbol in $K$-theory.  Indeed, in Section \ref{sec:Toeplitz-operators} we express the joint torsion of Toeplitz operators in terms of their tame symbols (Theorem \ref{thm:tame-symbols-L-infinty}).  This generalizes a result due to R.~Carey and J.~Pincus \cite
{Joint}: if $f,g \in H^\infty(S^1)$, then
\[
\tau(T_f, T_g) = \prod_{|a|<1} c_a(f,g)
\] 
Recently in \cite{Kaad-Nest3}, J.~Kaad and R.~Nest investigate the local behavior of joint torsion transition numbers associated to commuting tuples of operators.
They 
generalize the above Carey-Pincus formula and extend the notion of tame symbol to the setting of transversal functions on a complex analytic curve.

In \cite{Ehrhardt}, T.~Ehrhardt generalizes the Helton-Howe-Pincus formula by showing
\begin{equation} \label{eq:Ehrhardt}
e^A e^B - e^{A+B} \in \mathcal L^1
\end{equation}
whenever $[A,B] \in \mathcal L^1$, and moreover,
\[
\det \left( e^A e^B e^{-A-B} \right) = e^{\frac{1}{2} \tr [A,B]}
\]
Now let $P : L^2(S^1) \to H^2(S^1)$ be the orthogonal projection onto the Hardy space.  For any $\phi \in L^{\infty}(S^1)$ one may form the Toeplitz operator $T_\phi$, which is the compression to $H^2(S^1)$ of multiplication by $\phi$.  With $A = T_{(I-P) \phi}$ and $B = T_{P \phi}$, under suitable regularity assumptions, \eqref{eq:Ehrhardt} implies that
\begin{equation} \label{eq:Ehrhardt-Toeplitz}
T_{e^\phi} - e^{T_\phi} \in \mathcal L^1
\end{equation}

We investigate the following question: To what extent does \eqref{eq:Ehrhardt-Toeplitz} hold with the exponential replaced by more general functions?  In Section \ref{sec:Fredholm-modules} we consider entire functions in the more general setting of summable Fredholm modules.  Then in Section \ref{sec:Toeplitz-operators} we specialize to Toeplitz operators.  Theorem \ref{thm:holomorphic-functional-calculus-commutes} establishes \eqref{eq:Ehrhardt-Toeplitz} when
\begin{enumerate}
\item $f$ is holomorphic on a neighborhood of $\sigma(T_\phi)$, or
\item $\phi$ is real-valued and $f$ is $C^\infty$ on $\phi(S^1)$.
\end{enumerate}
Along the way, we investigate the functional calculus modulo ideals of compact operators.  Under suitable assumptions on $f$, we have
\begin{enumerate}
\item $f(A) - f(A') \in \mathcal L^p$ if $A-A' \in \mathcal L^p$ (Proposition \ref{prop:Lp-perturbation-holomorphic}).
\item $[f(A), B] \in \mathcal L^p$ if $[A,B] \in \mathcal L^p$ (see Proposition \ref{prop:commutator-trace-functional-calc}).
\item $T_{f(\phi)} - f(T_\phi) \in \mathcal L^p$ in a $2p$-summable Fredholm module (Proposition \ref{prop:L1-norm}).
\end{enumerate}
Thus we obtain functional calculi on the Calkin-type algebra $\mathcal B / \mathcal L^p$ of bounded operators $\mathcal B$ modulo the Schatten ideal $\mathcal L^p$ of compact operators with $p$-summable singular values.
Result (2) is due to A.~Connes \cite
{Connes}.  
We also obtain expressions for the trace and estimates on the $\mathcal L^p$-norms of operators as above.

Finally, we apply these results to obtain an integral formula for the joint torsion $\tau(T_f,T_g)$ of Toeplitz operators $T_f$ and $T_g$ as
\[
\exp \frac{1}{2\pi i} \left( \int_{S^1} \, \log f \, d(\log g) - \log g(p) \int_{S^1} \, d(\log f) \right)
\]
This formula was obtained by R.~Carey and J.~Pincus \cite
{Joint}, and previously by
J.~W.~Helton and R.~Howe \cite{Helton-Howe1} in an equivalent form.  See also \cite{Esnault-Viehweg}.


\subsection*{Acknowledgments}
This work benefited greatly from enlightening conversations with a number of people.  In particular I wish to thank
Richard Carey, 
Guillermo Corti\~{n}as,
Ra\'ul Curto, 
J\"{o}rg Eschmeier, 
Alexander Gorokhovsky, 
Nigel Higson, 
Jens Kaad, 
Jerry Kaminker, 
Matthias Lesch, 
Ryszard Nest,
Joel Pincus, 
and Mariusz Wodzicki.


\section{Preliminaries}  \label{sec:preliminaries}

\subsection{The determinant invariant}

For any unital ring $R$ and ideal $I$, there are algebraic $K$-groups $K_i(R)$, $K_i(R/I)$, and $K_i(R,I)$ that fit into Quillen's long exact sequence
\[
\dots \to K_{i+1}(R/I) \xrightarrow{\partial} K_i(R,I) \to K_i(R) \to K_i(R/I) \xrightarrow{\partial} \dots
\]
Any two commuting invertible elements $a,b \in R/I$ determine a Steinberg symbol (or Loday product, up to a sign) $\{a,b\} \in K_2(R/I)$.  Now let $R= \mathcal B = \mathcal B(H)$ be the algebra of bounded operators on a Hilbert space $H$, and let $I = \mathcal L^1 = \mathcal L^1(H)$ be the ideal of trace class operators on $H$.  Then the Fredholm determinant induces a map
\[
\det : K_1(\mathcal B, \mathcal L^1) \to \mathbf C^\times
\]
In fact, $K_1(\mathcal B, \mathcal L^1) = V \oplus \mathbf C^\times$ for a vector space $V$ with uncountable linear dimension, and $\det$ can be seen as the projection onto the second factor \cite{Anderson-Vaserstein}.  The following definition is due to L.~Brown \cite{Brown}:

\begin{definition}
Let $a,b \in \mathcal B / \mathcal L^1$ be invertible and commuting elements.  The determinant invariant $d(a,b)$ is
\[
d(a,b) = \det \partial_2 \{a,b\} \in \mathbf C^\times
\]
\end{definition}

The determinant of a multiplicative commutator remarkably depends only on the $K$-theory of the operators involved:

\begin{proposition}[\cite{Brown}]
If $A$ and $B$ are invertible operators with $[A,B] \in \mathcal L^1$, then
\[
\det(ABA^{-1}B^{-1}) = d(\pi(A), \pi(B) )
\]
where $\pi: \mathcal B \to \mathcal B / \mathcal L^1$ is the quotient map.
\end{proposition}

In fact, the determinant invariant can always be calculated in terms of a multiplicative commutator.  To see this, let $a,b \in \mathcal B / \mathcal L^1$ be invertible commuting elements, and pick lifts $A$ and $B$ in $\mathcal B$ of $a$ and $b$.  Let $S_A$ be an operator with index opposite that of $A$.  For example, we may take $S_A$ to be a unilateral shift or a parametrix for $A$.  Pick $S_B$ similarly.  Then $A \oplus S_A \oplus I$ has index zero, so we may pick a finite rank operator $F_A$ such that $\tilde A = A \oplus S_A \oplus I + F_A$ is invertible.  Similarly for $\tilde B = B \oplus I \oplus S_B + F_B$.

\begin{corollary}
With $\tilde A$ and $\tilde B$ as above, we have
\[
d(a,b) = \det ( \tilde A \tilde B \tilde A^{-1} \tilde B^{-1} )
\]
\end{corollary}

\subsection{Joint torsion}

In \cite{Joint}, R.~Carey and J.~Pincus introduce a notion of determinant known as joint torsion $\tau(A,B)$ associated to any pair of commuting Fredholm operators $A$ and $B$.  This invariant is defined as follows: the operator $A$ induces a morphism of the Koszul complex
\[
K_\bullet(B): \quad 0 \to H \xrightarrow{B} H \to 0
\]
The mapping cone $C(A)$ is identified with the joint Koszul complex $K_\bullet(A,B)$.  
This forms an exact triangle of complexes
\[
K_\bullet(B) \to K_\bullet(B) \to K_\bullet(A,B) \to
\]
and hence a long exact sequence $\mathcal E_A$ in homology.  By switching the roles of $A$ and $B$, we obtain a long exact sequence $\mathcal E_B$.

If $V$ is a finite dimensional vector space, let $\det V = \Lambda^{\dim V} V$.  
Associated to any exact sequence of finite dimensional vector spaces
\[
V: \quad 0 \to V_0 \to V_{1} \to \cdots V_n \to 0
\]
there is a canonical volume element
\[
\tau(V) \in \left( \det V_0^* \otimes \det V_1 \otimes \cdots \right)
\]
Using the canonical identification
\[
\det V \otimes \det V^* \cong \mathbf C
\]
we obtain the joint torsion $\tau(A,B)$, up to a sign, by comparing:
\[
\tau(\mathcal E_A) \otimes \tau(\mathcal E_B)^* \in \mathbf C^\times
\]
See \cite{Kaad-Nest3} for a discussion of joint torsion more generally for commuting morphisms of complexes.

In \cite{Perturbation}, joint torsion is extended to the situation when $A$ and $B$ do not necessarily commute, but satisfy $[A,B] \in \mathcal L^1$.
If $a,b \in \mathcal B / \mathcal L^1$ are invertible commuting elements, then by 
\cite{Eschmeier}, there exist lifts $A, D \in \mathcal B$ of $a$ and $B, C \in \mathcal B$ of $b$ such that
\[
AB=CD
\]
One may proceed as before and define long exact sequences
$\mathcal E_{A,D}$ and $\mathcal E_{B,C}$.
In this case however,
\[
\tau(\mathcal E_{A,D}) \otimes \tau(\mathcal E_{B,C})^* \in \det H(A) \otimes \det H(D)^* \otimes \det H(B)^* \otimes \det H(C)
\]
To obtain a scalar, Carey and Pincus introduce perturbation vectors
$\sigma_{A,D}$ and $\sigma_{B,C}$, 
which are canonical generators of the determinant lines, respectively,
\[
\det H(A) \otimes \det H(D)^*
\quad \text{and} \quad
\det H(B)^* \otimes \det H(C)
\]
We then obtain the joint torsion $\tau(A,B,C,D)$, up to a sign, by comparing:
\[
\tau(\mathcal E_{A,D}) \otimes \tau(\mathcal E_{B,C})^* 
\otimes \sigma_{A,D} \otimes \sigma_{B,C}
\in \mathbf C^\times
\]

Since joint torsion is equal to the determinant invariant \cite{JT-equals-DT}, 
we may write $\tau(A,B) = \tau(A',B',C',D')$, independent of choices of $A',B',C',D'$.  Moreover we have:

\begin{proposition} \label{prop:torsion-is-continuous}
Joint torsion is a continuous map into $\mathbf C$ from the space
\[
M = \{ (A,B) \, | \, A \text{ and } B \text{ are Fredholm and } [ A,B ] \in \mathcal L^1 \}
\]
endowed with the complete metric
\[
d((A_1, B_1),(A_2, B_2)) = \| A_1 - A_2 \| + \| B_1 - B_2 \| + \| [A_1, B_1] - [A_2, B_2] \|_1
\]
\end{proposition}

\subsection{Properties of joint torsion}

In this section we record a number of properties of joint torsion for later use.  
The following result expresses joint torsion as a multiplicative Lefschetz number.  This follows quickly from the definitions \cite{JT-equals-DT}.  See \cite{Reciprocity} for an earlier result on the determinant invariant.

\begin{lemma} \label{lemma:multiplicative-Lefschetz}
If $A$ and $B$ are commuting Fredholm operators with vanishing Koszul homology, then
\[
\tau(A,B) = \frac{\det B|_{\ker A}}{\det B|_{\coker A}} \frac{\det A|_{\coker B}}{\det A|_{\ker B}}
\]
\end{lemma}

\begin{lemma} \label{lemma:torsion-of-exponentials}
If $[A, B] \in \mathcal L^1$, then
\[
\tau(e^A, e^B) = e^{\tr [A, B]}
\]
\end{lemma}

\begin{proof}
In this case,
\[
\tau(e^A, e^B) = \det \left( e^A e^B e^{-A} e^{-B} \right)
\]
and the result follows by the Helton-Howe-Pincus formula.
\end{proof}


It is convenient to state the following variational formula using the logarithmic derivative.  Thus
$\frac{d}{dz} \log u$ 
should be interpreted as
$u^{-1} \frac{d}{dz} u$.

\begin{corollary} \label{cor:variational-exponentials}
Suppose $A(z)$ is a differentiable family of operators such that
$[A(z), B] \in \mathcal L^1$ for every $z$.  
If in addition $[A(z), B]$ is differentiable in $\mathcal L^1$, then
\[
\frac{d}{dz} \log \tau(e^{A(z)}, e^B)
= \log \tau (e^{\frac{d}{dz} A(z)}, e^B)
\]
\end{corollary}

\begin{lemma} \label{lemma:torsion-Steinberg-etc}
Whenever the following joint torsion numbers are defined, we have:
\begin{enumerate}
\item $\tau(A,B_1 B_2) = \tau(A,B_1) \cdot \tau(A,B_2)$
\item $\tau(A,I) = 1$
\item $\tau(A,B)^{-1} = \tau(B,A)$
\item $\tau(A, I-A) = 1$
\item $\tau(A,-A) = 1$
\item $\overline{\tau(A,B)} = \tau(A^*, B^*)^{-1}$
\item $\tau(A,B^{-1}) = \tau(A,B)^{-1}$
\item $\tau(A,A) = (-1)^{\ind \, A}$
\end{enumerate}
\end{lemma}

\begin{proof}
Properties (1)-(6) follow from the corresponding properties of the determinant invariant.  See for instance Lemma 4.2.14 and Theorem 4.2.17 of \cite{Rosenberg}.  
Property (7) follows from (1) and (2).  
To verify (8), notice that the two torsion factors in the definition of joint torsion are the same.  Thus we are left with $(-1)^{\nu(A,A)}$, where $\nu(A,A)$ is the sign in the definition of joint torsion.  The result follows since $\nu(A,A) = \ind \, A$.
\end{proof}

\begin{lemma} \label{lemma:values-of-joint-torsion}
Whenever the following joint torsion numbers are defined, we have:
\begin{enumerate}
\item $\tau(A,A^*) \in \mathbf R$.
\item If $A$ and $B$ are self-adjoint, then $|\tau(A,B)| = 1$.
\item If $B$ is an idempotent, i.e.~$B^2=B$, then $\tau (A,B) = 1$.
\item If $A$ is self-adjoint and $B$ is a partial isometry, then $\tau(A,B) \in \mathbf R$.
\item If $A$ and $B$ are partial isometries, then $|\tau(A,B)| = 1$.
\end{enumerate}
\end{lemma}

\begin{proof} \leavevmode
\begin{enumerate}
\item By properties (6) and (3) of Lemma \ref{lemma:torsion-Steinberg-etc},
\[
\overline{\tau(A,A^*)} = \tau (A^*,A)^{-1} = \tau(A,A^*)
\]

\item Since $A$ and $B$ are self-adjoint, Lemma \ref{lemma:torsion-Steinberg-etc}(6) implies that
\[
\overline{\tau(A,B)} = \tau(A,B)^{-1}
\]

\item By Lemma \ref{lemma:torsion-Steinberg-etc}(1), $\tau(A,B) = \tau(A,B)^2$, and the result follows since joint torsion is nonzero.

\item First let $T$ be any Fredholm operator which commutes with $B$ modulo $\mathcal L^1$.  Since $B$ is a Fredholm partial isometry, $T$ also commutes with $B^*$ modulo $\mathcal L^1$.  Since $B^*B$ is a projection, (3) implies that
\[
\tau(T^*, B^*) \cdot \tau(T^*,B) = \tau (T^*,B^*B) = 1
\]
so by Lemma \ref{lemma:torsion-Steinberg-etc}(6),
\begin{equation} \label{partial-isom-torsion}
\tau(T^*,B) = \overline{\tau(T,B)}
\end{equation}
The result follows by setting $T=A$ since $A^*=A$.

\item Applying \eqref{partial-isom-torsion} to both $A$ and $B$ yields
\[
\tau(A,B) = \tau(A^*,B^*)
\]
and the result follows by Lemma \ref{lemma:torsion-Steinberg-etc}(6). \qedhere
\end{enumerate}
\end{proof}

In (4), if $A$ is in fact positive, we will use the behavior of joint torsion under the functional calculus to show that $\tau(A,B) >0$ (Proposition \ref{prop:partial-isom-positive-torsion}).

\begin{lemma}
If $A$ and $B$ are commuting Fredholm operators, then for any $\lambda \ne 0$, 
\[
\tau(A, \lambda B) = \lambda^{\ind \, A - \dim H_0 + \dim H_2} \tau(A,B)
\]
where $H_0 = H_0(A,B)$ and $H_2 = H_2(A,B)$ are the joint Koszul homology spaces.
\end{lemma}

\begin{proof}
The two long exact sequences $\mathcal E_A$ and $\mathcal E_{\lambda B}$ in the definition of $\tau(A, \lambda B)$ are the same as those for $\tau(A,B)$, except for a factor of $\lambda$, given by the exponent on $\lambda$ above.
\end{proof}

\begin{corollary} \label{cor:variational-scalar-multiplication}
If $A$ and $B$ are commuting Fredholm operators, then
\[
\frac{d}{d\lambda} \log \tau (A, \lambda B) = \frac{\ind \, A - \dim H_0 + \dim H_2}{\lambda}
\]
\end{corollary}

\section{Transformation rules for joint torsion}  \label{sec:functional-calculus}

\subsection{Commutators}

If $[A, B] \in \mathcal L^p$, and either 
\begin{enumerate}
\item $f$ is holomorphic on a neighborhood of $\sigma(A)$, or 
\item $A$ is self-adjoint and $f$ is $C^\infty$ on $\sigma(A)$, 
\end{enumerate}
then $[f(A), B] \in \mathcal L^p$ \cite[Appendix 1]{Connes}.  
Below we calculate the trace of such a commutator.


\begin{lemma} \label{lemma:commutator-trace-entire}
If $[A,B] \in \mathcal L^1$ and $f$ is an entire function, then
\[
\tr [f(A), B] = \tr \left( f'(A) [A,B] \right)
\]
\end{lemma}

\begin{proof}
Write $f(z) = \sum c_k z^k$. Then $[f(A), B] = \sum c_k [A^k, B]$.
Using the identity
\begin{equation} \label{eq:commutator-powers}
[A^k, B] = \sum_{l=1}^{k} A^{l-1} [A,B] A^{k-l}
\end{equation}
we find that
\[
\tr [A^k, B] = \tr (k \, A^{k-1} [A,B])
\]
Hence
\begin{align*}
\tr [f(A), B]
&= \tr \sum k c_k A^{k-1} [A,B] \\
&= \tr \left( f'(A) [A,B] \right) \qedhere
\end{align*}
\end{proof}

Let $f$ be holomorphic on a neighborhood of the spectrum $\sigma(A)$ of an operator $A$.
By an admissible contour $\Gamma$ for defining $f(A)$, we mean a collection of Jordan curves in the neighborhood that enclose $\sigma(A)$ on the left.  Thus
\begin{equation} \label{eq:holomorphic-functional-calculus}
f(A) = \frac{1}{2 \pi i} \int_\Gamma (\lambda - A)^{-1} f(\lambda) \, d \lambda
\end{equation}

\begin{proposition} \label{prop:commutator-trace-functional-calc}
Suppose $[A,B] \in \mathcal L^1$.  If either
\begin{enumerate}
\item $f$ is holomorphic on a neighborhood of $\sigma(A)$, or
\item $A$ is self-adjoint and $f$ is $C^{\infty}$ on $\sigma(A)$
\end{enumerate}
then
\[
\tr [f(A), B] = \tr \left( f'(A) [A,B] \right)
\]
\end{proposition}

\begin{proof}
Recall that in both cases $[f(A),B] \in \mathcal L^1$ by \cite[Appendix 1]{Connes}, and we adapt arguments therein.
\begin{enumerate}
\item
Let $\Gamma$ be an admissible contour for defining $f(A)$.
Since $[(\lambda - A)^{-1}, B] = (\lambda - A)^{-1} [A,B] (\lambda - A)^{-1}$, we find
\[
[f(A), B] = \frac{1}{2\pi i} \int_\Gamma (\lambda - A)^{-1} [A,B] (\lambda - A)^{-1} f(\lambda) \, d \lambda
\]
and the mapping $\lambda \mapsto [(\lambda - A)^{-1}, B]$ is continuous into $\mathcal L^1$.  Moreover,
\[
\tr \left( (\lambda - A)^{-1} [A,B] (\lambda - A)^{-1} \right) = \tr \left( (\lambda - A)^{-2} [A,B] \right)
\]
Hence
\begin{align*}
\tr [f(A), B] 
&= \tr \left( \frac{1}{2\pi i} \int_\Gamma (\lambda - A)^{-2} f(\lambda) \, d \lambda \, [A,B] \right) \\
&= \tr \left( f'(A) [A,B] \right)
\end{align*}

\item
We may assume that $f$ has compact support, so that $f = \hat g$, the Fourier transform of a Schwartz class function $g$.  Hence
\[
[f(A), B] = \frac{1}{\sqrt{2\pi}} \int [e^{-itA}, B] g(t) \, dt
\]
By the preceding lemma,
\[
\tr [e^{-itA}, B] = \tr \left( -it e^{-itA} [A,B] \right)
\]
and again by continuity,
\begin{align*}
\tr [f(A), B] 
&=  \tr \frac{1}{\sqrt{2\pi}} \int -it e^{-itA} g(t) \, dt [A, B] \\
&= \tr \left( f'(A) [A,B] \right) \qedhere
\end{align*}
\end{enumerate}
\end{proof}

\begin{corollary}
With the same hypotheses as above,
\[
\tau (e^{f(A)}, e^B) = \tau(e^{A}, e^{f'(A) B})
\]
\end{corollary}

\begin{proof}
Since $A$ and $f'(A)$ commute, we have $f'(A)[A,B] = [A, f'(A)B]$.
The result then follows by Lemma \ref{lemma:torsion-of-exponentials}.
\end{proof}

\subsection{Perturbations}

Analogues of Lemma \ref{lemma:commutator-trace-entire} and Proposition \ref{prop:commutator-trace-functional-calc} hold for suitable functions applied to $\mathcal L^p$-perturbations.  We will need the following estimate for the exponential function:

%

\begin{proposition}
If $A$ and $A'$ are self-adjoint with $A-A' \in \mathcal L^p$, then
$e^{i t A} - e^{i t A'} \in \mathcal L^p$ 
with
\[
\| e^{i t A} - e^{i t A'} \|_p \leq C \left( |t| + 1 \right)
\]
where
\[
C = \max_{0 \leq t \leq 1} \| e^{it A} - e^{it A'} \|_p
\]
\end{proposition}

\begin{proof}
Using the identity
\begin{equation} \label{eq:difference-of-powers}
r^n - s^n = \sum_{k=1}^{n} s^{k-1} (r-s) r^{n-k}
\end{equation}
we find that
\[
\| e^{itn A} - e^{itn A'} \|_p \leq n \| e^{it A} - e^{it A'} \|_p
\]
The result then follows by scaling.
\end{proof}

\begin{proposition} \label{prop:functional-calculus-0-coset}
Let $K\in \mathcal L^p$.  If either
\begin{enumerate}
\item $f$ is holomorphic on a neighborhood of $\sigma(K)$, or
\item $K$ is self-adjoint and $f$ is $C^\infty$ on $\sigma(K)$,
\end{enumerate}
then $f(K) - f(0) I \in \mathcal L^p$.
\end{proposition}

Note that in (2), $\sigma(K)$ consists of 0 and real eigenvalues possibly accumulating to 0 by the spectral theorem for compact self-adjoint operators.

\begin{proof} \leavevmode
\begin{enumerate}
\item Let $\Gamma$ be an admissible contour for defining $f(K)$.  Then
\begin{align*}
f(K) - f(0)I
&= \int_\Gamma \left[ (\lambda-K)^{-1} - (\lambda I)^{-1} \right] f(\lambda) \, d\lambda \\
& = K \int_\Gamma (\lambda^2 - \lambda K)^{-1} f(\lambda) \, d\lambda
\end{align*}
The latter integral converges in norm, and the result follows.

\item We may assume that $f$ has compact support, so that $f=\hat g$ for a Schwartz class function $g$.  Then
\begin{align*}
f(K) - f(0)I
&= \int \left( e^{-i t K} - I \right) g(t) \, dt
\end{align*}
The integral converges in $\mathcal L^p$-norm by the preceding proposition with $A=K$ and $A' = 0$. \qedhere
\end{enumerate}
\end{proof}

\begin{proposition} \label{prop:Lp-perturbation-holomorphic}
Let $A-A' \in \mathcal L^p$.  If either
\begin{enumerate}
\item $f$ is holomorphic on a neighborhood of $\sigma(A) \cup \sigma(A')$ and there is a contour that defines both $f(A)$ and $f(A')$, or
\item $A$ and $A'$ are self-adjoint and $f$ is $C^\infty$ on 
$\sigma(A) \cup \sigma(A')$,
\end{enumerate}
then
$f(A) - f(A') \in \mathcal L^p$.
\end{proposition}

\begin{proof} \leavevmode
The proof proceeds as in the previous proposition.  For part (1), one uses the identity
\[
(\lambda-A)^{-1} - (\lambda-A')^{-1}
= (\lambda-A)^{-1} (A-A') (\lambda-A')^{-1} \qedhere
\]
\end{proof}

\subsection{Joint torsion}

For a given Fredholm operator $A$, we begin with a simple characterization of holomorphic functions $f$ that preserve the Fredholmness of $A$.  We will use the following factorization of holomorphic functions:

\begin{definition} \label{def:holomorphic-factorization}
Let $f$ be a holomorphic function on a neighborhood of a compact set $K$.  
Then the collection of zeros 
$\{ \lambda \in K \, | \, f(\lambda) = 0 \}$
is finite.  Define the polynomial 
\[
p_K(z) = \prod_{ \{ \lambda \in K \, | \, f(\lambda) = 0 \} } (z-\lambda)^{\text{ord}_\lambda(f)}
\]
where $\text{ord}_\lambda(f)$ is the order of the zero at $\lambda$.  Then
\[
f = p_K q_K
\]
for a holomorphic function $q_K$ with no zeros in $K$.
\end{definition}

The index formula \eqref{eq:Eschmeier-Putinar-n=1} below is a special case of \cite[Theorem 10.3.13]{Eschmeier-Putinar-Book}.  See also \cite[Theorem 1.1]{Kaad-Nest}.

\begin{proposition} \label{prop:Eschmeier-Putinar}
Let $A$ be a Fredholm operator and let $f$ be holomorphic on a neighborhood of $\sigma(A)$.  Then $f(A)$ is Fredholm if and only if $f^{-1}(0)$ is disjoint from the essential spectrum $\sigma_e(A)$.  In this case,
\begin{equation} \label{eq:Eschmeier-Putinar-n=1}
\ind \, f(A) = \sum_{ \{ \lambda \in \sigma(A) \, | \, f(\lambda) = 0 \} } \ord_\lambda(f) \cdot \ind (A - \lambda)
\end{equation}
\end{proposition}

\begin{proof}
Let $p=p_{\sigma(A)}$ and $q=q_{\sigma(A)}$ from the definition above.  Then $f(A) = p(A) q(A)$, and $q$ is invertible on a neighborhood of $\sigma(A)$, so $q(A)$ is invertible.  The first assertion then follows by factoring $p$, and the index formula follows by the additive property of the index: $\ind \, ST = \ind \, S + \ind \, T$.
\end{proof}

More generally one has the following necessary condition for the Borel functional calculus to preserve Fredholmness:

\begin{proposition} \label{prop:Borel-Fredholm-condition}
Let $A$ be a normal Fredholm operator, and let $f \in L^{\infty}(\sigma(A))$.  If the sets $f^{-1}(0)$ and $f^{-1}(\pm \infty)$ are finite and disjoint from $\sigma_{e}(A)$, then $f(A)$ is Fredholm.
\end{proposition}

\begin{proof}
The strategy is to excise the sets $f^{-1}(0)$ and $f^{-1}(\pm \infty)$ and use the resulting function to construct a parametrix for $f(A)$.
Suppose $f(\lambda) = 0$, $+\infty$, or $-\infty$.  Let $U_n$ be a nested sequence of open subsets of $\sigma(A)$ such that $\cap U_n = \{ \lambda \}$.  Then $\chi_{U_n}$ converges to $\chi_{\{ \lambda \} }$ pointwise, so $P_n = \chi_{U_n}(A)$ converges to $P = \chi_{\{ \lambda \} }(A)$ strongly.
Now $P$ is either $0$ or the projection onto the $\lambda$-eigenspace of $A$, which is finite dimensional since $A-\lambda$ is Fredholm.  Since $P_n$ is a descending sequence of projections that converge to a finite rank projection, there is an $N$ for which $P_n$ is finite dimensional for all $n>N$.

Let $U_\lambda = U_n$ and $\chi_\lambda = \chi_{U_\lambda}$ for some $n>N$.  By taking $n$ large enough, we may assume that the open sets $U_\lambda$ are pairwise disjoint, where $\lambda$ ranges over all the singularities and zeros of $f$.  Then
\[
g = (1 - \sum_\lambda \chi_\lambda ) f + \sum_\lambda \chi_\lambda
\]
is invertible in $L^{\infty}(\sigma(A))$, and $g(A) - f(A)$ is a finite rank operator.  Hence, $g(A)^{-1}$ is a parametrix for $f(A)$ modulo finite rank operators, so $f(A)$ is Fredholm.
\end{proof}


Next we obtain a multiplicative analogue of \eqref{eq:Eschmeier-Putinar-n=1}:

\begin{proposition} \label{prop:Eschmeier-Putinar-torsion}
Suppose $A$ and $B$ are Fredholm operators with $[A,B] \in \mathcal L^1$.  If $f$ is holomorphic on a neighborhood of $\sigma(A)$ and $f(A)$ is Fredholm, then
\[
\tau(f(A), B) = \prod_{ \{ \lambda \in \sigma(A) \, | \, f(\lambda) = 0 \} } \tau(A-\lambda, B)^{\ord_\lambda(f)} \cdot
\tau(q(A), B)
\]
with $q=q_{\sigma(A)}$ as in Definition \ref{def:holomorphic-factorization}, so that $q(A)$ is invertible.
\end{proposition}

\begin{proof}
First we note that $[f(A), B] \in \mathcal L^1$ by Proposition \ref{prop:commutator-trace-functional-calc}(1).
Writing $f = pq$, we have $[p(A), B] \in \mathcal L^1$, so $[q(A), B] \in \mathcal L^1$ as well.
By multiplicativity,
\[
\tau(f(A), B) = \tau(p(A),B) \cdot \tau(q(A),B)
\]
Since $p(A)$ is a product of factors $A-\lambda$, we find that $\tau(p(A),B)$ further factors as the product above.
\end{proof}

\subsection{Positivity of joint torsion}

In this section we investigate general conditions under which joint torsion is positive.  This is used to clarify the relationship between joint torsion and the polar decomposition, and also to obtain variational formulas.

\begin{proposition} \label{prop:partial-isom-positive-torsion}
Suppose $A$ and $B$ are Fredholm operators and $[A,B] \in \mathcal L^1$.
If $A$ is positive and $B$ is a partial isometry, then 
$\tau(A,B)>0$.
\end{proposition}

\begin{proof}
Let $F=P_{\ker A}$ be the orthogonal projection onto $\ker A = \im \, A^\perp$.  Then $A+ F$ is positive-definite.
By Proposition \ref{prop:commutator-trace-functional-calc}(2), $B$ commutes with $T= (A+F)^{1/2}$ modulo $\mathcal L^1$.  Hence
\[
\tau(A+ F, B) = \tau(T,B)^2
\]
By Lemma \ref{lemma:values-of-joint-torsion}, $\tau(T,B) \in \mathbf R$ since $T$ is self-adjoint.  Hence
\[
\tau(A,B) = \tau(A+ F, B) > 0 \qedhere
\]
\end{proof}

\begin{proposition}
Suppose $A$ and $B$ are Fredholm operators with $[A,B] \in \mathcal L^1$ and $[A,B^*] \in \mathcal L^1$.  Then with respect to the polar decompositions
\[
A = P_A V_A, \quad B = P_B V_B
\]
we have
\[
| \tau(A,B) | = \tau(P_A, V_B) \cdot \tau(V_A, P_B)
\]
and consequently,
\[
\frac{\tau(A,B)}{| \tau(A,B) |} = \tau(P_A, P_B) \cdot \tau(V_A, V_B)
\]
\end{proposition}

\begin{proof}
First notice that $P_A$ and $V_A$ are Fredholm since $A$ is.  Similarly, $P_B$ and $V_B$ are Fredholm.  We must show that the four joint torsion numbers above are well-defined, that is, the appropriate commutators lie in $\mathcal L^1$.
Our strategy is to show first that $[P_A,B]$, $[A,P_B]$, $[P_A, P_B] \in \mathcal L^1$, then $[V_A, P_B]$, $[P_A, V_B] \in \mathcal L^1$, and finally $[V_A, V_B]\in \mathcal L^1$.

If $F_{A} = P_{\ker A}$, then $A^*A+F_A$ is invertible and commutes with $B$ modulo $\mathcal L^1$.  
By Proposition \ref{prop:commutator-trace-functional-calc}(2), $[(A^*A+F_A)^{1/2}, B] \in \mathcal L^1$, so $[P_A, B] \in \mathcal L^1$ as well, with $P_A = (A^*A)^{1/2}$.  By reversing the roles of $A$ and $B$, we find that $[A,P_B] \in \mathcal L^1$.
Moreover, by replacing $B$ by $P_B$, we find that $[P_A,P_B] \in \mathcal L^1$.

Next, we calculate modulo $\mathcal L^1$:
\begin{align*}
[V_A, P_B]
& \equiv [(P_A+F_A)^{-1} P_A V_A, P_B] \\
& \equiv [(P_A+F_A)^{-1} A, P_B] \\
& \equiv (P_A+F_A)^{-1} [A, P_B] + [(P_A+F_A)^{-1}, P_B] A
\end{align*}
The first term is in $\mathcal L^1$ since $[A, P_B] \in \mathcal L^1$, and the second term is in $\mathcal L^1$ since $[P_A, P_B] \in \mathcal L^1$ as well.  Similarly, $[P_A, V_B] \in \mathcal L^1$ by reversing the roles of $A$ and $B$.

Again we calculate modulo $\mathcal L^1$:
\begin{align*}
[V_A, V_B]
\equiv{}& [(P_A+F_A)^{-1} A, (P_B+F_B)^{-1} B] \\
\equiv{}& (P_A+F_A)^{-1} [A, (P_B+F_B)^{-1}] B 
+ (P_A+F_A)^{-1} (P_B+F_B)^{-1} [A, B] \\
& + [(P_A+F_A)^{-1}, (P_B+F_B)^{-1} ] B A 
+ (P_B+F_B)^{-1} [(P_A+F_A)^{-1}, B] A
\end{align*}
As before, all four of the above terms are evidently in $\mathcal L^1$.

By the multiplicative property of joint torsion,
\[
\tau(A,B) = \tau(P_A, V_B) \cdot \tau(V_A, P_B) \cdot \tau(P_A, P_B) \cdot \tau(V_A, V_B)
\]
The first two factors are positive by preceding proposition.  
The third factor has magnitude one by Lemma \ref{lemma:values-of-joint-torsion}(2), as does the last factor by Lemma \ref{lemma:values-of-joint-torsion}(5).
\end{proof}

\begin{proposition} \label{prop:torsion-powers-of-positives}
Suppose $A$ and $B$ are Fredholm operators with $[A, B] \in \mathcal L^1$.  If $A$ is positive, and $B$ is a partial isometry, then for all $t \geq 0$,
\[
\tau(A^t, B) = \tau(A,B)^t
\]
If $A$ is positive-definite, then the formula holds for all $t \in \mathbf R$.
\end{proposition}

\begin{proof}
First we note that $\tau(A,B) > 0$ by Proposition \ref{prop:partial-isom-positive-torsion}, $[A^t, B]\in \mathcal L^1$ by Proposition \ref{prop:commutator-trace-functional-calc}(2), and $A^t$ is Fredholm with parametrix $(A+P_{\ker A})^{-t}$.
The formula holds for positive integers by repeated application of Lemma \ref{lemma:torsion-Steinberg-etc}(1), and for $t=0$ by Lemma \ref{lemma:torsion-Steinberg-etc}(2).
The formula also holds for all positive rational numbers: if $p$ and $q$ are any positive integers, then
\[
\tau(A^{p/q}, B)^q = \tau(A,B)^p
\]

If $F = P_{\ker A}$, then $A+F$ is positive and invertible, and 
$\tau(A^t, B) = \tau((A+F)^t, B)$.
We will show that the map $t \mapsto ((A+F)^t,B)$, $t > 0$, is a continuous map into the space $M$ in Proposition \ref{prop:torsion-is-continuous}.
Since $t \mapsto (A+F)^t$ is continuous in norm, it suffices to show that
\[
\lim_{t \to 0} \| [(A+F)^t, B] \|_1 = 0
\]
Since $[\log(A+F), B] \in \mathcal L^1$ by Proposition \ref{prop:commutator-trace-functional-calc}(2), this follows from the estimate
\[
\| [(A+F)^t, B] \|_1
\leq t \, e^{ t \| \log (A+F) \| } \| [\log (A+F), B] \|_1
\]

Joint torsion is continuous on $M$ by Proposition \ref{prop:torsion-is-continuous}, so the map $t \mapsto \tau(A^t, B)$ is continuous.  Thus the result extends from rational $t$ to all $t\geq 0$.  
Finally, if $A$ is positive definite, then $A^{-t}$ is also positive definite for any $t>0$.  By the above result for positive $t$, we find
\[
\tau(A^{-t},B)^t = \tau(A,B) \qedhere
\]
\end{proof}

A similar result holds when $A$ and $B$ are positive.  In this case, $\tau(A,B) \in S^1$ by Lemma \ref{lemma:values-of-joint-torsion}.
Suppose $A$ and $B$ are positive Fredholm operators with $[A,B] \in \mathcal L^1$.
If $F_A = P_{\ker A}$ and $F_B = P_{\ker B}$, then
\[
\phi(A,B) = - i \, \tr \, [\log(A+F_A), \log(B+F_B)] \in \mathbf R
\]
is well-defined by Proposition \ref{prop:commutator-trace-functional-calc}.  Since $\tau(A,B) = \tau(A+F_A, B+F_B)$, Lemma \ref{lemma:multiplicative-Lefschetz} gives 
\[
\tau(A,B) = e^{i \phi(A,B)}
\]
and we find that $\phi(A,B)$ enjoys the additive versions of the properties in Lemma \ref{lemma:torsion-Steinberg-etc}.
Moreover, we have:

\begin{proposition}
If $A$ and $B$ are positive Fredholm operators with $[A,B] \in \mathcal L^1$, then for all $t>0$,
\[
\tau(A^t,B) = e^{it\phi(A,B)} = \tau(A,B)^t
\]
If $A$ is positive-definite, then the formula holds for all $t \in \mathbf R$.
\end{proposition}

\begin{proof}
This follows by noticing that $\tau(A^t,B) = \tau((A+F_A)^t, B+F_B)$, then using the fact that $A+F_A$ and $B+F_B$ have logarithms.
\end{proof}

\begin{corollary} \label{cor:variational-powers}
Suppose $A$ and $B$ are Fredholm operators with $[A,B] \in \mathcal L^1$.
If $A$ is positive and $B$ is either positive or a partial isometry, then
\[
\frac{d}{dt} \log \tau(A^t, B) = \log \tau (A,B)
\]
\end{corollary}

\section{Fredholm modules}  \label{sec:Fredholm-modules}

Let $(A, H, F)$ be a $2p$-summable Fredholm module, 
i.e.~$[\phi,F] \in \mathcal L^{2p}$ for any $\phi \in A$.
Let $P=\frac{1}{2}(F+I)$ be the projection onto the $+1$-eigenspace of $F$, so in particular, $[\phi,P] \in \mathcal L^{2p}$ for any $\phi \in A$.

\begin{definition}
For $\phi \in A$, write $T_\phi = P \phi P$.
\end{definition}

The main goal of this section is to show that 
$f(T_\phi) - T_{f(\phi)} \in \mathcal L^p$ 
for suitable functions $f$.
First let us prove a corresponding result for the continuous functional calculus modulo compact operators:

\begin{proposition} \label{prop:continuous-functional-calculus-commutes}
Let $T \in \mathcal B$ be normal and let $\pi: \mathcal B \to \mathcal B / \mathcal K$ be the quotient map onto the Calkin algebra.  If $f \in C(\sigma(T))$, then $\pi(f(T)) = f(\pi(T))$.
\end{proposition}

\begin{proof}
The polynomial functional calculus commutes with the quotient map, so the result follows from the Stone-Weierstrass Theorem by approximating $f$ by polynomials.
\end{proof}

In the case of Toeplitz operators, we have the following:

\begin{corollary}
If $\phi \in C(S^1)$ and $f \in C(\sigma(T_\phi))$, then $f(T_\phi) - T_{f\circ \phi} \in \mathcal K$.  In particular, $T_{\phi} = \phi (T_z)$ modulo $\mathcal K$.
\end{corollary}

\begin{example}
If $f \in C(S^1)$ is non-vanishing, then
\[
\ind \, T_f = \ind \, f(T_z)
\]
If $f$ has a holomorphic extension to a neighborhood of the closed unit disk, then Proposition \ref{prop:Eschmeier-Putinar} yields the classical index formula for Toeplitz operators:
\[
\ind \, T_f = \int_{S^1} \frac{df}{f} = - ( \text{the winding number of } f )
\]
\end{example}

Next we show that the entire functional calculus commutes with the symbol map modulo $\mathcal L^{2p}$.
This complements the results of \cite{Ehrhardt} and Proposition \ref{prop:L1-norm} below.
Assume that $A$ is closed under the entire functional calculus.
Otherwise, we may replace $A$ by the algebra generated by $f(a)$, for all $a\in A$ and entire functions $f$.
The resulting algebra still has the property that $[\phi,P] \in \mathcal L^{2p}$.
In fact, if $[a,P] \in \mathcal L^{2p}$ and $f$ is holomorphic on a neighborhood of $\sigma(a)$, then 
$[f(a), P] \in \mathcal L^{2p}$
by \cite[Appendix 1]{Connes}.

\begin{lemma} \label{lemma:one-commutator}
For any $\phi \in A$ and integer $k>1$, $T_\phi^k - T_{\phi^k} \in \mathcal L^{2p}$, with
\[
\| T_\phi^k - T_{\phi^k} \|_{2p} \leq \frac{k(k-1)}{2} \| \phi \|^{k-1} \| [\phi, P] \|_{2p}
\]
\end{lemma}

\begin{proof}
Each term in the identity
\[
(P \phi)^k P - P \phi^k P = \sum_{l=1}^{k-1} (P\phi)^{k-l} [ \phi^l, P ] P
\]
contains a commutator, so $(P \phi)^k P - P \phi^k P \in \mathcal L^{2p}$.  Using the identity \eqref{eq:commutator-powers}, we estimate
\[
\| [ \phi^l, P ] \|_{2p} \leq l \| \phi \|^{l-1} \| [ \phi, P ] \|_{2p}
\]
Hence
\[
\| (P \phi)^k P - P \phi^k P \|_{2p} 
\leq \sum_{l=1}^{k-1} l \| \phi \|^{k-1} \| [\phi, P] \|_{2p}
\]
and the result follows.
\end{proof}

\begin{definition}
For an entire function $f(z) = \sum c_k z^k$, let $\tilde f(z) = \sum |c_k| z^k$.
\end{definition}

\begin{proposition} \label{prop:L2-norm}
For any $\phi \in A$ and any entire function $f$, $T_{f(\phi)} - f(T_\phi) \in \mathcal L^{2p}$ with
\[
\| T_{f(\phi)} - f(T_\phi) \|_{2p} \leq \frac{ \| [\phi, P] \|_{2p}}{2 \| \phi \|} \tilde f''(\| \phi \|)
\] 
\end{proposition}

\begin{proof}
Write $f(z) = \sum c_k z^k$.
The first two terms in the expansion
\begin{align*}
T_{f(\phi)} - f(T_\phi)
&= \sum_{k=0}^{\infty} c_k \left( P \phi^k P - (P\phi)^k P \right)
\end{align*}
vanish, and by Lemma \ref{lemma:one-commutator} we estimate
\begin{align*}
\| T_{f(\phi)} - f(T_\phi) \|_{2p}
&\leq \sum_{k=2}^{\infty} |c_k| \frac{k(k-1)}{2} \| \phi \|^{k-1} \| [\phi, P] \|_{2p} \\
&\leq \frac{ \| [\phi, P] \|_{2p}}{2 \| \phi \|} \sum_{k=2}^{\infty} k(k-1) |c_k| \| \phi \|^{k-2} \\
&\leq \frac{ \| [\phi, P] \|_{2p}}{2 \| \phi \|} \tilde f''(\| \phi \|) \qedhere
\end{align*}
\end{proof}

In fact, the entire functional calculus commutes with the symbol map modulo $\mathcal L^p$.  First we isolate the following analogue of Lemma \ref{lemma:one-commutator}:

\begin{lemma} \label{lemma:two-commutators}
For any $\phi \in A$ and integer $k>1$, $T_\phi^k - T_{\phi^k} \in \mathcal L^p$, with
\[
\| T_\phi^k - T_{\phi^k} \|_p \leq \frac{k(k-1)}{2} \| \phi \|^{k-2} \| [\phi,P] \|_{2p}^2
\]
\end{lemma}

\begin{proof}
First one verifies that
\[
(P \phi)^k P - P \phi^k P = \sum_{l=1}^{k-1} P [P, \phi^l] [P,\phi] (P \phi)^{k-l-1} P
\]
using the identity
$P [P, \psi] [P, \chi] P = P \psi (P-I) \chi P$.  
Each term of the sum contains a product of commutators, so it is in $\mathcal L^p$, and by \eqref{eq:commutator-powers},
\[
\| [P, \phi^l] \|_{2p}
\leq l \| \phi \|^{l-1} \| [\phi,P] \|_{2p}
\]
Hence
\[
\| (P \phi)^k P - P \phi^k P \|_p
\leq \sum_{l=1}^{k-1} l  \| \phi \|^{k-2} \| [\phi,P] \|_{2p}^2
\]
and the result follows.
\end{proof}

\begin{proposition} \label{prop:L1-norm}
For any $\phi\in A$ and any entire function $f$, $T_{f(\phi)} - f(T_\phi) \in \mathcal L^p$ with
\[
\| T_{f(\phi)} - f(T_\phi) \|_p \leq \frac{1}{2} \| [\phi,P] \|_{2p}^2 \tilde f''( \| \phi \|)
\] 
\end{proposition}

\begin{proof}
Write $f(z) = \sum c_k z^k$.
The first two terms in the expansion
\begin{align*}
T_{f(\phi)} - f(T_\phi)
&= \sum_{k=0}^{\infty} c_k \left( P \phi^k P - (P\phi)^k P \right)
\end{align*}
vanish, and by Lemma \ref{lemma:two-commutators} we estimate
\begin{align*}
\| T_{f(\phi)} - f(T_\phi) \|_p
&\leq \sum_{k=2}^{\infty} |c_k| \frac{k(k-1)}{2} \| [\phi,P] \|_{2p}^2 \| \phi \|^{k-2} \\
&= \frac{1}{2} \| [\phi,P] \|_{2p}^2 \sum_{k=2}^{\infty} k(k-1)  |c_k| \| \phi \|^{k-2} \\
&= \frac{1}{2} \| [\phi,P] \|_{2p}^2 \tilde f''( \| \phi \| ) \qedhere
\end{align*}
\end{proof}


We will need a sharper estimate for the exponential function:

\begin{proposition} \label{prop:difference-exponentials-trace-class}
If $\phi$ is self-adjoint, then $e^{T_{i t \phi}} - T_{e^{i t \phi}} \in \mathcal L^p$ for any $t \in \mathbf R$, and
\[
\| e^{T_{i t \phi}} - T_{e^{i t \phi}} \|_p \leq (|t|+1)^2 (c_1 + c_2)
\]
where
\[
c_1 = \max_{0 \leq t \leq 1} \| e^{T_{i t \phi}} - T_{e^{i t\phi}} \|_p
\quad \text{and} \quad
c_2 = \max_{0 \leq t \leq 1} \| [P,e^{it \phi}] \|_{2p}^2
\]
\end{proposition}

\begin{proof}
Setting $r = e^{T_{i t \phi}}$ and $s = T_{e^{i t\phi}}$ in identity \eqref{eq:difference-of-powers}, 
we obtain
\[
\| e^{T_{i n t \phi}} - (T_{e^{i t\phi}})^n \|_p \leq \| e^{T_{i t \phi}} - T_{e^{i t\phi}} \|_p \sum_{k=1}^{n} \| T_{e^{i t\phi}} \|^{k-1} \| e^{T_{i t \phi}} \|^{n-k} 
\]
Since $\| e^{T_{i t \phi}} \| = 1$ and $\| T_{e^{i t\phi}} \| \leq 1$, we find
\begin{equation} \label{eq:Connes1}
\| e^{T_{i n t \phi}} - (T_{e^{i t\phi}})^n \|_p \leq n \| e^{T_{i t \phi}} - T_{e^{i t\phi}} \|_p
\end{equation}
By Lemma \ref{lemma:two-commutators},
\[
\| (T_f)^n - T_{f^n} \|_p \leq \frac{n(n-1)}{2} \| [P,f] \|_{2p}^2 \| f \|^{n-2}
\]
Setting $f = e^{it\phi}$, we have  $\| f \| = 1$, so
\begin{equation} \label{eq:Connes2}
\| (T_{e^{it\phi}})^n - T_{e^{int \phi}} \|_p \leq \frac{n(n-1)}{2} \| [P,e^{it \phi}] \|_{2p}^2
\end{equation}
Combining \eqref{eq:Connes1} and \eqref{eq:Connes2}, we find
\[
\| e^{T_{i n t \phi}} - T_{e^{int \phi}} \|_p \leq n^2 ( \| e^{T_{i t \phi}} - T_{e^{i t\phi}} \|_p + \| [P,e^{it \phi}] \|_{2p}^2 )
\]
and the result follows by scaling.
\end{proof}

We are now able to obtain an analogue of Proposition \ref{prop:Lp-perturbation-holomorphic} for summable Fredholm modules.  Below we regard $\phi$ as an operator on $H$ and $T_\phi$ as an operator on $PH$, and we view $\sigma(\phi)$, $\sigma(T_\phi)$, $f(\phi)$, and $f(T_\phi)$ accordingly.

\begin{theorem} \label{thm:holomorphic-functional-calculus-commutes-module}
If either
\begin{enumerate}
\item $f$ is holomorphic on a neighborhood of $\sigma(\phi) \cup \sigma(T_\phi)$ and there is a contour $\Gamma$ that defines both $f(\phi)$ and $f(T_\phi)$, or
\item $\phi$ is self-adjoint and $f$ is $C^\infty$ on $\sigma(\phi) \cup \sigma(T_\phi)$,
\end{enumerate}
then $f(T_\phi) - T_{f(\phi)} \in \mathcal L^p$.
\end{theorem}

\begin{proof} \leavevmode
\begin{enumerate}
\item Notice that $(\lambda - P \phi P)^{-1} - P (\lambda - \phi)^{-1} P$ can be written as 
\[
P[(\lambda - \phi)^{-1}, P] [\phi, P] P (\lambda - P\phi P)^{-1}
\]
Since $[P,\phi] \in \mathcal L^{2p}$, the assignment
\[
\lambda \mapsto (\lambda - P \phi P)^{-1} - P (\lambda - \phi)^{-1} P
\]
is a continuous map into $\mathcal L^p$.  Hence
\[
f(T_\phi) - T_{f(\phi)} = \frac{1}{2\pi i} \int_\Gamma \left( (\lambda - P \phi P)^{-1} - P (\lambda - \phi)^{-1} P \right) f(\lambda) \, d\lambda
\]
converges in $\mathcal L^p$.

\item As in Proposition \ref{prop:functional-calculus-0-coset}, we may assume that $f$ has compact support, so that $f=\hat g$ for a Schwartz class function $g$. Then
\[
f(T_\phi) - T_{f(\phi)} = \frac{1}{\sqrt{2\pi}} \int \left( e^{T_{-i t \phi}} - T_{e^{-i t \phi}} \right) g(t) \, dt
\]
and the result follows by Proposition \ref{prop:difference-exponentials-trace-class}. \qedhere
\end{enumerate}
\end{proof}

\section{Toeplitz operators and tame symbols}  \label{sec:Toeplitz-operators}

In this section, we apply our techniques to Toeplitz operators and obtain formulas for joint torsion in terms of Tate tame symbols.  
Let $P: L^2(S^1) \to H^2(S^1)$ be the orthogonal projection onto the Hardy space $H^2(S^1)$.  Any function $\phi \in L^\infty(S^1)$ defines a bounded operator on $L^2(S^1)$ by multiplication by $\phi$.  
Let us begin by recalling results on commutators of Toeplitz operators.

\begin{lemma} \label{lemma:2-norm-commutator}
If $\phi \in L^\infty(S^1)$ is in the Sobolev space  $W^{\frac{1}{2}, 2}(S^1) = H^{\frac{1}{2}}(S^1)$, then 
\[
(I-P) \phi P, P \phi (I-P), [\phi,P] \in \mathcal L^2(L^2(S^1))
\]
with
\[
\| [\phi,P] \|_2 \leq \| \phi \|_{W^{\frac{1}{2}, 2}(S^1)}
\]
\end{lemma}

\begin{proof}
Write $\phi = \sum c_n e^{in\theta}$.
A straightforward calculation shows that 
$(I-P)\phi P \in \mathcal L^2$, with
\[
\| (I-P)\phi P \|_2^2 = \sum_{n > 0} n |c_n|^2,
\]
By taking adjoints, $P\phi(I-P) \in \mathcal L^2$ as well, with
\[
\| P \phi (I-P) \|_2^2 = - \sum_{n < 0} n |c_n|^2
\]
Hence
$[\phi,P] = (I-P)\phi P - P \phi (I-P) \in \mathcal L^2$, and
\[
\| [\phi,P] \|_2^2 = \sum_{n \neq 0} |n| |c_n|^2 \qedhere
\]
\end{proof}

In this case, Toeplitz operators have trace class commutators, and the Berger-Shaw formula calculates this trace:

\begin{theorem} \label{thm:Berger-Shaw}
If $f,g \in L^\infty(S^1) \cap W^{\frac{1}{2}, 2}(S^1)$, then $[T_f, T_g] \in \mathcal L^1$. If $f,g \in C^1(S^1)$, then
\[
\tr [T_f, T_g] = \frac{1}{2 \pi i} \int f \, dg
\]
\end{theorem}

\begin{proof}
First notice that
$[T_f, T_g] = Pg(I-P)fP - Pf(I-P)gP$.
Both terms are trace class since they are products of two operators which are Hilbert-Schmidt by the preceding lemma.
The trace formula then follows by writing $f$ and $g$ in the basis $\{ e^{in\theta} \}$.
\end{proof}

\subsection{$H^\infty$ symbols}

\begin{proposition} \label{prop:torsion-invertible}
Suppose $\phi \in C(S^1) \cap H^{\infty}(S^1)$ is invertible in $H^{\infty}(S^1)$.
\begin{enumerate}
\item If $|\lambda|>1$, then $\tau(T_\phi, T_z - \lambda) = 1$.
\item If $|\lambda|<1$, then $\tau(T_\phi, T_z - \lambda) = \phi(\lambda)$,
with $\phi$ extended holomorphically to the interior of the unit disk.
\end{enumerate}
\end{proposition}

\begin{proof}
First notice that $T_{\phi}$ is invertible with inverse $T_{1/\phi}$.
If $|\lambda|>1$, then $z-\lambda$ is invertible in $H^{\infty}(S^1)$ as well.
The operators $T_\phi$ and $T_z-\lambda$ commute, so in this case
$\tau(T_\phi, T_z - \lambda) = 1$.

Now suppose $| \lambda| < 1$.
By Lemma \ref{lemma:torsion-Steinberg-etc}(6), it is enough to show that
$\tau(T_{\bar z} - \bar \lambda, T_{\bar \phi}) = \overline{\phi(\lambda)}$.
In this case, $\coker (T_{\bar z} - \bar \lambda) = \{ 0 \}$ and
\[
\ker (T_{\bar z} - \bar \lambda) = \text{span} \left( \frac{1}{1 - \bar \lambda z} = \sum_{k = 0}^\infty (\bar \lambda z)^k \right)
\]
The operator $T_{\bar \phi}$ acts as multiplication by $\overline{\phi(\lambda)}$ on the one dimensional subspace $\ker (T_{\bar z} - \bar \lambda)$.  In particular,
\[
\det T_{\bar \phi} |_{\ker (T_{\bar z} - \bar \lambda)} = \overline{\phi(\lambda)}
\]
This is the joint torsion by Lemma \ref{lemma:multiplicative-Lefschetz} since $T_{\bar \phi}$ is invertible and commutes with $T_{\bar z} - \bar \lambda$.
\end{proof}

\begin{proposition} \label{prop:single-tame-symbol}
Let $\lambda, \mu \in \mathbf C$.
\begin{enumerate}
\item If $| \lambda_1 | > 1$ and $| \lambda_2 | > 1$, then 
$\tau(T_z-\lambda_1, T_z-\lambda_2) = 1$.

\item If $| \lambda_1 | < 1$ and $| \lambda_2 | > 1$, then 
$\tau(T_z-\lambda_1, T_z-\lambda_2) = (\lambda_1 - \lambda_2)^{-1}$.

\item If $| \lambda_1 | > 1$ and $| \lambda_2 | < 1$, then 
$\tau(T_z-\lambda_1, T_z-\lambda_2) = \lambda_2 - \lambda_1$.

\item If $| \lambda_1 | < 1$ and $| \lambda_2 | < 1$, then 
$\tau(T_z-\lambda_1, T_z-\lambda_2) = -1$.
\end{enumerate}
\end{proposition}

\begin{proof}
In case (1), both $T_z - \lambda_1$ and $T_z - \lambda_2$ are invertible in $H^{\infty}(S^1)$ and commute with each other, so $\tau(T_z-\lambda_1, T_z-\lambda_2) = 1$.

Cases (2) and (3) follow from the preceding proposition.

For (4), we use the multiplicative property of joint torsion:
\begin{align*}
\tau(T_z-\lambda_1, T_z-\lambda_2)
={}& \tau(T_{\bar z}, T_{\bar z}) \cdot
\tau(T_{\bar z}, T_{\bar z (z- \lambda_2)})^{-1} \\
& \cdot \tau(T_{\bar z (z- \lambda_1)}, T_{\bar z})^{-1} \cdot
\tau(T_{\bar z ( z - \lambda_1 ) }, T_{\bar z (z- \lambda_2)})
\end{align*}
The first term is $-1$, and the last term is $1$ since $T_{\bar z ( z - \lambda_1 ) }$ and $T_{\bar z (z- \lambda_2)}$ are invertible and commute with each other.
The middle two terms are both 1 by the above proposition.
Hence, $\tau(T_z-\lambda_1, T_z-\lambda_2) = -1$.
\end{proof}

If $f$ and $g$ are meromorphic at $\lambda \in \mathbf C$, then the quotient
\[
\frac{f^{\text{ord}_\lambda(g)}}{g^{\text{ord}_\lambda(f)}}
\]
is regular at $\lambda$.  Here, $\text{ord}_\lambda$ denotes the order of the zero or pole at $\lambda$.

\begin{definition} \label{def:tame-symbol}
The tame symbol $c_\lambda(f,g)$ of $f$ and $g$ at $\lambda$ is defined as
\[
c_\lambda(f,g) = (-1)^{\text{ord}_\lambda(f) \cdot \text{ord}_\lambda(g)} \frac{f^{\text{ord}_\lambda(g)}}{g^{\text{ord}_\lambda(f)}} (\lambda)
\]
\end{definition}

\begin{definition}
If $a \in \mathbf C$ is nonzero, the Blaschke factor $B_a$ is
\[
B_a(z) = \frac{|a|}{a} \frac{a-z}{1-\bar a z}
\]
Let $B_0(z) = z$ and $B_\infty(z) = \bar z$.  A product of Blaschke factors is known as a Blaschke product.
\end{definition}

Notice that for $z \in S^1$, we have
\begin{equation} \label{eq:conjugate-Blaschke}
\overline{B_a(z)} = B_{\bar a}(\bar z) = B_{1/\bar a}(z)
\end{equation}

The preceding propositions may be rephrased in terms of tame symbols, and in fact we have:

\begin{proposition} \label{prop:tame-symbols-cases}
Suppose $f$ and $g$ are products of
\begin{enumerate}
\item invertible functions in $C(S^1) \cap H^\infty(S^1)$,
\item polynomials, and
\item Blaschke factors $B_a$ with $|a|<1$.
\end{enumerate}
If $f$ and $g$ are non-vanishing on $S^1$, then
\[
\tau(T_f, T_g) = \prod_{|\lambda|<1} c_{\lambda}(f,g)
\]
\end{proposition}

\begin{proof}
A straightforward calculation with $f(z) = z-\lambda_1$ and $g(z) = z - \lambda_2$ verifies that
\[
\prod_{|\lambda_i|<1} c_{\lambda_i}(z-\lambda_1, z-\lambda_2)
\]
agrees with (1)-(4) in Proposition \ref{prop:single-tame-symbol}.  Since both joint torsion and the tame symbol are multiplicative, the result holds for polynomials.  
By Proposition \ref{prop:torsion-invertible}, we find that the result holds for factors of type (1) and (2).  If $|a|<1$, then $B_a$ is the product of a polynomial and
$(1-\bar a z)^{-1} \in H^{\infty}(S^1)$.  
Hence factors of type (3) are products of types (1) and (2).
\end{proof}

We will need the following Beurling-Szeg\H{o} factorization into inner and outer functions.  See for instance \cite{Colwell}.

\begin{theorem} \label{thm:Beuerling-Szego}
If $f \in H^{\infty}(S^1)$ is continuous and non-vanishing on $S^1$, then there exists an outer function $\phi$ that is invertible in $H^{\infty}(S^1)$ such that
\[
f = \phi \cdot \prod B_a
\]
where the above product is taken over finitely many zeros $a$ with $|a|<1$.
\end{theorem}

The following result was first obtained in \cite[Proposition 1]{Joint}.  See also \cite{Park}.  See \cite{Kaad-Nest3} for a generalization to the multivariable setting.

\begin{theorem} \label{thm:tame-symbols}
If $f,g \in H^{\infty}(S^1)$ are continuous and non-vanishing on $S^1$, then $\tau(T_f, T_g)$ is the product of tame symbols:
\[
\tau(T_f, T_g) = \prod_{|a|<1} c_a(f,g)
\]
\end{theorem}

\begin{proof}
As in the preceding theorem, write
\[ 
f = \phi_f \cdot \prod B_{a}, \quad g = \phi_g \cdot \prod B_{b}
\]
By Proposition \ref{prop:tame-symbols-cases}, the joint torsion numbers
\[
\tau(\phi_f, \phi_g), \tau(\phi_f, B_{b}), \tau(B_{a}, \phi_g), \tau(B_{a}, B_{b})
\]
agree with the corresponding tame symbols.  The result then follows since both joint torsion and the tame symbol are bimultiplicative.
\end{proof}

\subsection{$L^\infty$ symbols}

In this section we extend the above result to the almost commuting setting. 

\begin{proposition} \label{prop:non-commutative-tame-symbol-invertible}
Suppose $\phi \in C(S^1) \cap H^\infty(S^1)$ is invertible in $H^\infty(S^1)$.
\begin{enumerate}
\item If $|\lambda|>1$, then $\tau(T_\phi, T_{\bar z} - \lambda) = \frac{\phi(1/\lambda)}{\phi(0)}$.
\item If $|\lambda|<1$, then $\tau(T_\phi, T_{\bar z} - \lambda) = \frac{1}{\phi(0)}$.
\end{enumerate}
\end{proposition}

\begin{proof}
If $\lambda = 0$, then
\[
\tau( T_\phi, T_{\bar z} ) \cdot \tau( T_\phi, T_{z} )
= \tau( T_\phi, I )
= 1
\]
By Proposition \ref{prop:torsion-invertible}, the second factor is $\phi(0)$, so the result follows in this case.

If $\lambda \neq 0$, we may write
\[
- \frac{1}{\lambda} ( \bar z - \lambda ) z = z - \frac{1}{\lambda}
\]
so that
\[
\tau(T_\phi, T_{-1/\lambda}) 
\cdot \tau(T_\phi, T_{\bar z -\lambda}) 
\cdot \tau(T_\phi, T_z)
= \tau(T_\phi, T_{z-1/\lambda})
\]
The first factor is 1 since $\phi$ is invertible and the third factor is $\phi(0)$.  Hence
\[
\tau(T_\phi, T_{\bar z -\lambda}) = \frac{\tau(T_\phi, T_{z-1/\lambda})}{ \phi(0) }
\]
and result follows by the Proposition \ref{prop:torsion-invertible}.
\end{proof}

\begin{proposition} \label{prop:non-commutative-single-tame-symbol}
Let $\lambda, \mu \in \mathbf C$.
\begin{enumerate}
\item If $| \lambda_1 | > 1$ and $| \lambda_2 | > 1$, then 
$\tau(T_z-\lambda_1, T_{\bar z}-\lambda_2) = 1 - ( \lambda_1 \lambda_2 )^{-1}$.

\item If $| \lambda_1 | < 1$ and $| \lambda_2 | > 1$, then 
$\tau(T_z-\lambda_1, T_{\bar z}-\lambda_2) = - \lambda_2^{-1}$.

\item If $| \lambda_1 | > 1$ and $| \lambda_2 | < 1$, then 
$\tau(T_z-\lambda_1, T_{\bar z}-\lambda_2) = - \lambda_1^{-1}$.

\item If $| \lambda_1 | < 1$ and $| \lambda_2 | < 1$, then 
$\tau(T_z-\lambda_1, T_{\bar z}-\lambda_2) = (\lambda_1 \lambda_2 - 1)^{-1}$.
\end{enumerate}
\end{proposition}

\begin{proof}
This result follows from Proposition \ref{prop:single-tame-symbol}, as the preceding proposition follows from Proposition \ref{prop:torsion-invertible}.
\end{proof}

Notice that $\bar z - \lambda$ extends meromorphically to the interior of the unit disk as
\[
\frac{1}{z} - \lambda
\]
with a simple pole at $0$ and a simple zero at $\frac{1}{\lambda}$.
A straightforward verification shows that the previous two propositions express the joint torsion as a product of tame symbols.
Since a Blaschke factor is the ratio of two linear factors, we have the following non-commutative generalization of Proposition \ref{prop:tame-symbols-cases}:

\begin{proposition} \label{prop:non-commutative-tame-symbols-cases}
Suppose $f$ and $g$ are products of
\begin{enumerate}
\item invertible functions in $C(S^1) \cap H^\infty(S^1)$,
\item trigonometric polynomials in $z$ and $\bar z$, and
\item Blaschke factors $B_a$ with $a \in \mathbf C \cup \{ \infty \}$.
\end{enumerate}
If $f$ and $g$ are non-vanishing on $S^1$, then
\[
\tau(T_f, T_g) = \prod_{|\lambda|<1} c_{\lambda}(f,g)
\]
Here, $f$ and $g$ have been extended meromorphically to the interior of the unit disk.
\end{proposition}

Now suppose $f \in L^\infty(S^1)$ such that $T_f$ is Fredholm.
Then $f$ is continuous and non-vanishing on $S^1$, say with winding number $n$.  
The function $z^{-n} f(z)$ has winding number zero, so there is a continuous function $\tilde f$ such that
\[
e^{\tilde f(z)} = z^{-n} f(z)
\]
Let $\tilde f_+ = P \tilde f$, $\tilde f_- = (I-P) \tilde f$, where $P:L^2(S^1) \to H^2(S^1)$ is the orthogonal projection as usual.  Then
\begin{equation} \label{eq:holomorphic-antiholomorphic}
f(z) = z^{-n} e^{\tilde f_-} e^{\tilde f_+}
\end{equation}
Thus we may write
$f = f_- f_+$
with $f_+, \bar f_- \in H^{\infty}$ continuous and non-vanishing.
By Theorem \ref{thm:Beuerling-Szego}, we can write
\[
f_+ = f_1 \cdot \prod B_a, \quad f_- = f_2 \cdot \prod B_b
\]
where $f_1, \bar f_2 \in H^\infty$ are invertible in $H^\infty$, 
the zeros $a$ satisfy $|a|<1$,
and the zeros $b$ satisfy $|b|>1$.  
Letting $f_0$ be the product of Blaschke factors above, we have the factorization
\[
f = f_0 f_1 f_2
\]

If $f$ is smooth, then so is $z^{-n}f$, and we can take $\tilde f$ to be smooth as well.  Consequently $\tilde f_+$ and $\tilde f_-$ are smooth, for example because the projection $P$ can be expressed in terms of the Hilbert transform, which preserves regularity.  
Hence the related functions $f_\pm, f_i$, $i=0,1,2$, are smooth as well.  
Define $g_i$, $i=0,1,2$, similarly.  
As in Proposition \ref{prop:Eschmeier-Putinar-torsion}, we see that joint torsion factors as a discrete part (tame symbols) and a continuous part (a determinant):

\begin{theorem} \label{thm:tame-symbols-L-infinty}
If $f,g \in C^{\infty}(S^1)$ are non-vanishing on $S^1$, then
\[
\tau(T_f, T_g) = \prod_{|a|<1} 
c_a(f_0 f_1, g_0 g_1) 
\cdot \frac{ \overline{ c_a( \bar g_0, \bar f_2) } }{ \overline{ c_a( \bar f_0, \bar g_2) } }
\cdot \frac{ \tau(T_{f_1}, T_{g_2}) }{ \tau(T_{g_1}, T_{f_2}) }
\]
Here $f_i$, $g_i$ are as above, and
\[
\tau(T_{f_1}, T_{g_2}) 
= \exp \left( \frac{1}{2 \pi i} \int \log f_1 \, d ( \log g_2 ) \right)
\]
for continuous choices of logarithms of $f_1$ and $g_2$, 
and similarly for $\tau(T_{g_1}, T_{f_2})$.
\end{theorem}

\begin{proof}
By the multiplicative property of joint torsion, we find that
\[
\tau(T_f, T_g) = \tau(T_{f_0 f_1}, T_{g_0 g_1}) \cdot \tau(T_{f_0 f_1}, T_{g_2}) 
\cdot \tau(T_{f_2}, T_{g_0 g_1}) \cdot \tau(T_{f_2}, T_{g_2})
\]
The first factor is the product of tame symbols by Theorem \ref{thm:tame-symbols}.
The fourth factor is 1 since $T_{f_2}$ and $T_{g_2}$ are invertible commuting operators.  
Next we calculate the second factor;
the third factor is dealt with similarly.
Again using multiplicativity, the second factor is
\[
\tau(T_{f_0}, T_{g_2}) \cdot \tau(T_{f_1}, T_{g_2})
\]

For the first factor, notice that $\bar f_0$ is still a Blaschke product and $\bar g_2 \in H^\infty$ is invertible in $H^\infty$.  
Hence $\tau(T_{f_0}, T_{g_2}) =\overline{ \tau(T_{\bar f_0}, T_{\bar g_2})^{-1} }$ by Lemma \ref{lemma:torsion-Steinberg-etc}(6), and the latter is $\overline{ c_a( \bar f_0, \bar g_2)^{-1} }$ by Proposition \ref{prop:non-commutative-tame-symbols-cases}.
In the second factor, both $T_{f_1}$ and $T_{g_2}$ are invertible.  Hence their joint torsion is the multiplicative commutator 
\[
\det \left( T_{f_1} T_{g_2} T_{f_1^{-1}} T_{g_2^{-1}} \right) 
\]
which is calculated as the claimed integral by the Helton-Howe-Pincus formula \ref{lemma:torsion-of-exponentials} and the Berger-Shaw formula \ref{thm:Berger-Shaw}.
\end{proof}

\subsection{An integral formula}

Now we apply and refine the results of Section \ref{sec:Fredholm-modules} in the case of Toeplitz operators.  
Let $L^2 = L^2(S^1)$ and $H^2 = H^2(S^1)$.  
%
%
Recall that 
if $\phi \in L^\infty(S^1)$, the spectrum of the multiplication operator $\phi \in \mathcal B(L^2)$ is the essential range of $\phi$.
If $\phi \in C(S^1)$, the spectrum of the Toeplitz operator $T_\phi \in \mathcal B(H^2)$ consists of $\phi(S^1)$, together with the connected components of $\mathbf C - \phi(S^1)$ about which $\phi$ has nonzero winding number.
See for instance 
\cite[Chapter 7]{Douglas}.  
As a consequence, we obtain the following:


%

\begin{theorem} \label{thm:holomorphic-functional-calculus-commutes}
Let $\phi \in L^\infty(S^1) \cap W^{\frac{1}{2},2}(S^1)$.
If either
\begin{enumerate}
\item $f$ is holomorphic on a neighborhood of $\phi(S^1)$, or
\item $\phi$ is real-valued and $f$ is $C^\infty$ on $\phi(S^1)$,
\end{enumerate}
then $f(T_\phi) - T_{f\circ \phi} \in \mathcal L^1(H^2)$.
\end{theorem}

\begin{proof}
Let $\Gamma$ be an admissible contour for defining $f(T_\phi) \in \mathcal B(H^2)$, as in \eqref{eq:holomorphic-functional-calculus}.  By the above discussion, we see that $\Gamma$ can also be used to define $f(\phi) \in \mathcal B(L^2)$, that is,
\[
f(\phi) = \frac{1}{2\pi i} \int_\Gamma (\lambda - \phi)^{-1} f(\lambda) \, d\lambda
\]
The result then follows by Theorem \ref{thm:holomorphic-functional-calculus-commutes-module} with $p=1$.
\end{proof}

We conclude with an illustration of the above results by deriving an integral formula for the joint torsion of Toeplitz operators \cite[Theorem 7]{Joint}.  An equivalent formula was previously obtained in \cite{Helton-Howe1}.  See also \cite{Esnault-Viehweg}.

\begin{theorem}
If $f,g \in C^\infty(S^1)$ are non-vanishing functions, then
\[
\tau(T_f, T_g) =
\exp \frac{1}{2\pi i} \left( \int_{S^1} \, \log f \, d(\log g) - \log g(p) \int_{S^1} \, d(\log f) \right)
\]
\end{theorem}

The integrals are taken counterclockwise starting at any point $p = e^{i \alpha} \in S^1$.  If $h(e^{i\theta}) = |h(e^{i\theta})| e^{i\phi(\theta)}$ for a continuous function $\phi: [\alpha, \alpha + 2 \pi] \to \mathbf R$, then we take $\log h (e^{i\theta} ) = \log |h| + i \phi(\theta)$.  Any other choice of $\log h$ will differ by a multiple of $2\pi i$ and hence will leave the quantity in the theorem unaffected.

\begin{proof}
Let $n$ and $m$ be the winding numbers of $f$ and $g$, respectively.
Define $\tilde f$, $\tilde f_+$, and $\tilde f_-$ as in \eqref{eq:holomorphic-antiholomorphic}, and similarly for $g$.
By Theorem \ref{thm:holomorphic-functional-calculus-commutes}, $T_{e^{\tilde f}} = e^{T_{{\tilde f}}}$ modulo $\mathcal L^1$, so we find
\[
\tau(T_f, T_g) = \tau(T_z, T_z)^{mn} \cdot \tau(T_z, T_{e^{\tilde g}})^n \cdot \tau(T_{e^{\tilde f}}, T_z)^m \cdot \tau(e^{T_{\tilde f}}, e^{T_{\tilde g}})
\]
The first factor is $(-1)^{mn}$ by Proposition \ref{prop:single-tame-symbol}.  By applying Proposition \ref{prop:torsion-invertible} with $\lambda = 0$ and both $\phi = e^{\tilde g_+}$ and $\phi = e^{\tilde g_-}$, we find that the second term is 
$e^{-n \tilde g_+(0)}$.  
Similarly, the third term is
$e^{m \tilde f_+(0)}$.
By Lemma \ref{lemma:torsion-of-exponentials} and Theorem \ref{thm:Berger-Shaw}, the fourth term is
\[
\exp \left( \frac{1}{2 \pi i} \int \tilde f \, d \tilde g \right)
\]
Hence
\begin{equation} \label{eq:preliminary-integral-formula}
\tau(T_f, T_g) = \exp \left( \pi i m n + m \tilde f_+(0) - n \tilde g_+(0) + \frac{1}{2 \pi i} \int \tilde f \, d \tilde g \right)
\end{equation}


Now we calculate the last term in the exponential:
\begin{align*}
\int \tilde f \, d \tilde g
& = \int \log(e^{-in \theta} f) \, d(\log( e^{-i m \theta} g) \\
& = \int -i n \theta \, d \tilde g 
+ \int \log f \, d ( \log (e^{-i m \theta} g ) )
\end{align*}
Integration by parts gives
\[
\int -i n \theta \, d \tilde g = 
- i n \theta \tilde g |_{\alpha}^{\alpha + 2 \pi}
+ \int i n \tilde g \, d \theta
\]
The first term is $-2 \pi i n \tilde g(p)$ since $\tilde g$ has winding number zero.  By writing $\tilde g$ in terms of the orthonormal basis elements $e^{ik\theta}$, we see that the second term is $2 \pi i n \tilde g_+(0)$.  Next we calculate
\[
\int \log f \, d ( \log ( e^{-i m \theta} g ) )
= \int \tilde f \cdot -i m \, d \theta
+ \int i n \theta \cdot -i m \, d \theta
+ \int \log f \, d ( \log g )
\]
As before the first term is $-2 \pi i m \tilde f_+(0)$, and the second term is $2 mn \pi^2 + 2 \pi mn \alpha$.
Combining this with \eqref{eq:preliminary-integral-formula} gives
\[
\tau(T_f, T_g) = \exp \left( - n \tilde g(p) - i mn \alpha + \frac{1}{2\pi i} \int \log f \, d ( \log g ) \right)
\]
The first term is
\[
-n ( -i m \alpha + \log g(p) )
= i mn \alpha - \frac{1}{2 \pi i} \log g(p) \int d (\log f)
\]
Hence
\[
\tau(T_f, T_g) =
\exp \frac{1}{2\pi i} \left( \int_{S^1} \, \log f \, d(\log g) - \log g(p) \int_{S^1} \, d(\log f) \right) \qedhere
\]
\end{proof}

\printbibliography

\end{document}